\documentclass[preprint,12pt]{elsarticle}
\usepackage{amsthm,amsmath,amssymb}
\usepackage[colorlinks=true,citecolor=black,linkcolor=black,urlcolor=blue]{hyperref}
\usepackage{graphicx}
\usepackage{float}
\usepackage{mathtools}
\usepackage{epstopdf}
\usepackage{epsfig}
\usepackage{subfigure}

\theoremstyle{plain}
\newtheorem{theorem}{Theorem}
\newtheorem{lemma}[theorem]{Lemma}
\newtheorem{corollary}[theorem]{Corollary}
\newtheorem{proposition}[theorem]{Proposition}

\theoremstyle{definition}
\newtheorem{definition}[theorem]{Definition}

\newtheorem{conjecture}[theorem]{Conjecture}

\newtheorem{question}[theorem]{Question}

\theoremstyle{remark}
\newtheorem{remark}[theorem]{Remark}

\title{Twist polynomials of delta-matroids}
\author{Qi Yan\\
\small School of Mathematics\\[-0.8ex]
\small China University of Mining and Technology\\[-0.8ex]
\small P. R. China\\
Xian'an Jin\footnote{Corresponding author.}\\
\small School of Mathematical Sciences\\[-0.8ex]
\small Xiamen University\\[-0.8ex]
\small P. R. China\\
\small{\tt Email:qiyan@cumt.edu.cn; xajin@xmu.edu.cn}
}
\date{}

\begin{document}
\begin{abstract}
Recently, Gross, Mansour and Tucker introduced the partial duality polynomial
of a ribbon graph and posed a conjecture that there is no orientable ribbon graph
whose partial duality polynomial has only one non-constant term. We found an infinite family of counterexamples for the conjecture and showed that essentially these are the only counterexamples. This is also obtained independently by Chumutov and Vignes-Tourneret and they posed a problem: it would be interesting to know whether the partial duality polynomial and the related conjectures would make sence for general delta-matroids. In this paper, we show that partial duality polynomials have delta-matroid analogues. We introduce the twist polynomials of delta-matroids and discuss its basic properties for delta-matroids. We give a characterization of even normal binary delta-matroids whose twist polynomials have only one term and then prove that the twist polynomial of a normal binary delta-matroid contains non-zero constant term if and only if its intersection graph is bipartite.
\end{abstract}
\maketitle

\section{Introduction}
For any ribbon graph $G$, there is a natural dual ribbon graph $G^{*}$, also called {\it geometric dual}.
Chmutov \cite{CG} introduced an extension of geometric duality called {\it partial duality}. Roughly
speaking, a partial dual $G^{A}$ is obtained by forming the geometric dual with respect to only a subset $A\subseteq E(G)$ of a ribbon graph $G$.

In \cite{GMT}, Gross, Mansour and Tucker introduced the enumeration of the partial duals $G^A$ of a ribbon graph $G$, by Euler genus $\varepsilon$, over all edge subsets $A\subseteq E(G)$. The associated generating functions, denoted as $^{\partial}\varepsilon_{G}(z)$, are called \emph{partial duality polynomials} of $G$. They formulated the following conjecture.

\begin{conjecture}\label{con1}\cite{GMT}
 There is no orientable ribbon graph having a non-constant partial duality polynomial with only one non-zero coefficient.
\end{conjecture}

The conjecture is not true. In \cite{QYXJ} we found an infinite family of counterexamples. Furthermore, we  \cite{QYXJ2} proved that essentially these are the only counterexamples. This is also obtained independently by Chumutov and Vignes-Tourneret in \cite{SCFV} and they also posed the following question:

\begin{question}\cite{SCFV}
Ribbon graphs may be considered from the point of view of delta-matroid. In this way the concepts of partial duality and genus can be interpreted in terms of delta-matroids \cite{CISR, CMNR}. It would be interesting to know whether the partial duality polynomial and the related conjectures would make sence for general delta-matroids.
\end{question}

In this paper, we show that partial duality polynomials have delta-matroid analogues. We introduce the twist polynomials of delta-matroids and discuss its basic properties and consider Conjecture \ref{con1} for delta-matroids. We give a characterization of even normal binary delta-matroids whose twist polynomials  have only one term and then prove that the twist polynomial of a normal binary delta-matroid contains non-zero constant term if and only if its intersection graph is bipartite.



\section{Preliminaries}

\subsection{Delta-matroids}


 A \emph{set system} is a pair $D=(E, \mathcal{F})$, where $E$ or $E(D)$, is a finite set,  called the \emph{ground set}, and $\mathcal{F}$, or $\mathcal{F}(D)$, is a collection of  subsets of $E$, called \emph{feasible sets}.  A set system $D$ is \emph{proper} if
$\mathcal{F}\neq \emptyset$ and $D$ is \emph{trivial} if $E=\emptyset$.

As introduced by Bouchet in \cite{AB1}, a \emph{delta-matroid} is a proper set system $D=(E, \mathcal{F})$ for which  satisfies the \emph{symmetric exchange axiom}: for all triples $(X, Y, u)$ with $X, Y \in \mathcal{F}$ and $u\in X\Delta Y$, there is a $v\in X\Delta Y$ (possibly  $v=u$ ) such that $X\Delta \{u, v\}\in \mathcal{F}$. Here $X\Delta Y=(X\cup Y)\backslash (X\cap Y)$  is the usual symmetric difference of sets.

Let $D=(E, \mathcal{F})$ be a delta-matroid. If for any $A_{1}, A_{2}\in \mathcal{F}$, we have $|A_{1}|=|A_{2}|$. Then $D$ is said to be a \emph{matroid} and we refer to $\mathcal{F}$ as its \emph{bases}. If a set system forms a matroid $M$, then we usually denote $M$ by $(E, \mathcal{B})$. The \emph{rank function} of $M$ takes any subset $A\subseteq E(M)$ to the number
$$\max\{|A\cap B|: B\in \mathcal{B}\}.$$
This is written as $r_{M}(A)$. We say that the \emph{rank} of $M$, written $r(M)$, is equal to $|B|$ for any $B\in\mathcal{B}(M)$. It is clear that the rank function of a matroid $M$ on a set $E$ has the following properties \cite{JO}:
\begin{enumerate}
  \item If $X\subseteq Y\subseteq E$, then $r_{M}(X)\leq r_{M}(Y)$;
  \item If $X$ and $Y$ are subsets of $E$, then $$r_{M}(X\cup Y)+r_{M}(X\cap Y)\leq r_{M}(X)+r_{M}(Y).$$
\end{enumerate}
The \emph{nullity} of $A$, written $n_{M}(A)$,  is $|A|-r_M(A)$.

Bouchet \cite{AB4} defined an analogue of the rank function for delta-matroids. Let $D=(E, \mathcal{F})$ be a delta-matroid. For $A\subseteq E$, define $$\rho_{D}(A):=|E|-min\{|A\Delta F|:  F\in \mathcal{F}\}.$$

A delta-matroid is \emph{even} if for every pair $F$ and $\widetilde{F}$ of its feasible sets $|F\Delta\widetilde{F}|$ is even.
Otherwise, we call the delta-matroid \emph{odd}. A delta-matroid is \emph{normal} if the empty set is feasible.

For a delta-matroid $D=(E, \mathcal{F})$, let $\mathcal{F}_{max}(D)$ and $\mathcal{F}_{min}(D)$ be the collection of maximum and minimum cardinality feasible sets of $D$, respectively. Bouchet \cite{AB2} showed that the set systems $D_{max}=(E, \mathcal{F}_{max})$ and $D_{min}=(E, \mathcal{F}_{min})$ are matroids. The  \emph{width} of $D$, denote by $w(D)$, is defined by $$w(D):=r(D_{max})-r(D_{min}).$$
Particularly, $D$ is a matroid if and only if $w(D)=0$.

A fundamental operation on delta-matroids, introduced by Bouchet in \cite{AB1}, is the twist. Let $D=(E, \mathcal{F})$ be a set system. For $A\subseteq E$, the \emph{twist} of $D$ with respect to $A$, denoted by $D*A$, is given by $$(E, \{A\Delta X: X\in \mathcal{F}\}).$$ The \emph{dual} of $D$, written $D^{*}$, is equal to $D*E$. Using the identity $$(A\Delta C)\Delta(B\Delta C)=A\Delta B,$$ it is straightforward to show that the twist of a delta-matroid is a delta-matroid \cite{AB1}. Note that being even is preserved under taking twists.

\begin{definition}
The  \emph{twist polynomial} of any delta-matroid $D=(E, \mathcal{F})$ is the generating function
$$^{\partial}w_{D}(z):=\sum_{A\subseteq E}z^{w(D*A)}$$
that enumerates all twists of $D$ by width.
\end{definition}

\begin{definition}\cite{CMNR}
For delta-matroids $D=(E, \mathcal{F})$ and $\widetilde{D}=(\widetilde{E}, \widetilde{\mathcal{F}})$ with $E\cap \widetilde{E}=\emptyset$, the  \emph{direct sum} of $D$ and $\widetilde{D}$, written $D\oplus \widetilde{D}$, is the delta-matroid defined as $$D\oplus \widetilde{D}:=(E\cup \widetilde{E}, \{F\cup \widetilde{F}: F\in \mathcal{F}~\text{and}~\widetilde{F}\in \widetilde{\mathcal{F}}\}).$$
\end{definition}

A delta-matroid is  \emph{disconnected} if it can be written as $D\oplus \widetilde{D}$ for some non-trivial delta-matroids $D$ and $\widetilde{D}$, and \emph{connected} otherwise.

Let $D=(E, \mathcal{F})$ be a proper set system. An element $e\in E$ contained in every feasible set of $D$ is said to be a \emph{coloop}, while an element $e\in E$ contained in no feasible set of $D$ is said to be a \emph{loop}.

Let $D=(E, \mathcal{F})$ be a proper set system and $e\in E$. Then $D$ \emph{delete} by $e$, denoted $D\backslash e$, is defined as $D\backslash e:=(E\backslash e, \mathcal{F}')$, where
\[\mathcal{F}':=\left\{\begin{array}{ll}
\{F: F\in \mathcal{F}, F\subseteq E\backslash e\}, & \text{if $e$ is not a coloop}\\

\{F\backslash e: F\in \mathcal{F}\}, & \text{if $e$ is a coloop}.
\end{array}\right.\]
Bouchet \cite{AB1} has shown that the order in which deletions are performed does not matter.
Let $A\subseteq E$. We define $D\setminus A$ as the result of deleting every element of $A$ in any order. The \emph{restriction} of $D$ to $A$, written $D|_{A}$, is the delta-matroid $D\setminus (E\backslash A)$. Throughout the paper, we will often omit the set brackets in the case of a single element set. For example, we write $D*e$ instead of $D*\{e\}$, or $D|_{e}$ instead of $D|_{\{e\}}$.

\subsection{Ribbon graphs}
We give a brief review of ribbon graphs referring the reader to \cite{EM1,EM} for further details.

\begin{definition}[\cite{EM}]
A \emph{ribbon graph} $G=(V(G), E(G))$ is a (possibly non-orientable) surface with boundary
represented as the union of two sets of discs, a set $V(G)$ of vertices, and a set $E(G)$ of edges
such that
\begin{enumerate}
\item The vertices and edges intersect in disjoint line segments;
\item Each such line segment lies on the boundary of precisely one vertex and precisely one edge;
\item Every edge contains exactly two such line segments.
\end{enumerate}
\end{definition}

A bouquet is a ribbon graph with only one vertex. An edge $e$ of a ribbon graph is a \emph{loop} if it is incident with exactly one vertex. A loop is \emph{non-orientable} if together with its incident vertex it forms a M\"{o}bius band, and is \emph{orientable} otherwise.
A \emph{signed rotation} of a bouquet is a cyclic ordering of the half-edges at the vertex and if the edge is an orientable loop, then we give the same sign $+$ to the corresponding two half-edges, and give the different signs (one $+$, the other $-$) otherwise. The sign $+$ is always omitted. See Figure 1 for an example.
\begin{figure}[!htbp]
\begin{center}
\includegraphics[width=9cm]{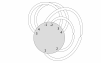}
\caption{The signed rotation of the bouquet is $(-1, -2, 3, 4, 2, 1, 3, 4)$.}
\label{f01}
\end{center}
\end{figure}

\begin{definition}\cite{CMNR}
Let $G=(V, E)$ be a ribbon graph and let $$\mathcal{F}:=\{F\subseteq E(G): \text{$F$ is the edge set of a spanning quasi-tree of $G$}\}.$$ We call $D(G)=:(E, \mathcal{F})$ the delta-matroid of $G$.
\end{definition}


\subsection{Binary and intersection graphs}

For a finite set $E$, let $C$ be a symmetric $|E|$ by $|E|$ matrix over $GF(2)$, with rows and columns indexed, in the same order, by the elements of $E$. Let $C[A]$ be the principal submatrix of $C$ induced by the set $A\subseteq E$. We define the set system $D(C)=(E, \mathcal{F})$ with $$\mathcal{F}:=\{A\subseteq E: C[A] \mbox{ is non-singular}\}.$$  By convention $C[\emptyset]$ is non-singular. Then $D(C)$ is a delta-matroid \cite{AB4}. A delta-matroid is said to be \emph{binary} if it has a twist that is isomorphic to $D(C)$ for some symmetric matrix $C$ over $GF(2)$.

Let $D=(E, \mathcal{F})$ be a delta-matroid. If $D=D(C)$ for some symmetric matrix $C$ over $GF(2)$, that is, $D$ is a normal binary delta-matroid, then we can get $C$ as following \cite{Moff}:
\begin{enumerate}
\item Set $C_{v, v}=1$ if and only if $\{v\}\in \mathcal{F}$. This determines the diagonal entries of $C$;
\item Set $C_{u,v}=1$ if and only if $\{u\}, \{v\}\in \mathcal{F}$ but $\{u, v\}\notin \mathcal{F}$, or $\{u, v\}\in \mathcal{F}$ but $\{u\}$ and $\{v\}$ are not both in $\mathcal{F}$. Then the feasible sets of size two determine the off-diagonal entries of $C$.
\end{enumerate}

Let $D=(E, \mathcal{F})$ be a normal binary delta-matroid. Then there exists a symmetric $|E|$ by $|E|$ matrix $C$ over $GF(2)$ such that $D=D(C)$. The \emph{intersection graph} $G_{D}$ of $D$ is the graph with vertex set $E$ and in which two vertices $u$ and $v$ of $G_{D}$ are adjacent if and only if $C_{u, v}=1$. Recall that a \emph{looped simple graph} is a graph obtained from a simple graph by adding (exactly) one loop to some of its vertices.  If $D$ is odd, then $G_{D}$ is a looped simple graph, and if $D$ is even, then $G_{D}$ is a simple graph. Note that $D$ is connected if and only if $G_{D}$ is connected.

Conversely, the adjacency matrix $A(G)$ of a looped simple graph $G$ is the matrix over $GF(2)$ whose rows and columns correspond to the vertices of $G$; and where, $A(G)_{u, v}=1$ if and only if $u$ and $v$ are adjacent in $G$ and $A(G)_{u, u}=1$ if and only if there is a loop at $u$.
Let $D$ be a normal binary delta-matroid. It obvious that $D=D(A(G_{D}))$.


\section{Main results}
\begin{proposition}\label{pro 2}
Let $D=(E, \mathcal{F})$ and $\widetilde{D}=(\widetilde{E}, \widetilde{\mathcal{F}})$ be two delta-matroids and $A\subseteq E$. Then
\begin{enumerate}
  \item $^{\partial}w_{D}(1)=2^{|E|}$;
  \item $^{\partial}w_{D}(z)=~^{\partial}w_{D*A}(z);$
  \item $^{\partial}w_{D\oplus \widetilde{D}}(z)=~^{\partial}w_{D}(z)~^{\partial}w_{\widetilde{D}}(z).$
\end{enumerate}
\end{proposition}

\begin{proof}
For (1), the evaluation $^{\partial}w_{D}(1)$ counts the total number of twists, which is $2^{|E|}$.
For (2), this is because the sets of all twists of $D$ and $D*A$ are the same.
For (3), the underlying phenomenon is the additivity of  width over the direct sum. It follows immediately that for any subset $B\subseteq E\cup \widetilde{E}$, we have
$$(D\oplus \widetilde{D})*B=D*(B\cap E)\oplus \widetilde{D}*(B\cap \widetilde{E}),$$
from which formula (3) follows.
\end{proof}

\begin{remark}
By Proposition \ref{pro 2}, we can observe that analyzing the twist polynomials of all delta-matroids  is equivalent to analyzing
normal delta-matroids. Consequently, it is natural to focus on normal delta-matroids.
\end{remark}

\begin{lemma}\cite{CMNR}\label{lem 1}
Let $G=(V, E)$ be a ribbon graph, $A\subseteq E$ and $e\in E$. Then $D(G^{A})=D(G)*A$ and $\varepsilon(G)=w(D(G))$.
\end{lemma}

\begin{lemma}\label{lem 3}
Let $G=(V, E)$ be a ribbon graph. Then $^{\partial}w_{D(G)}(z)=~^{\partial}\varepsilon_{G}(z)$.
\end{lemma}

\begin{proof}
By Lemma \ref{lem 1}, for any $A\subseteq E$, $$w(D(G)*A)=w(D(G^{A}))=\varepsilon(G^{A}).$$
Hence $^{\partial}w_{D(G)}(z)=~^{\partial}\varepsilon_{G}(z)$.
\end{proof}

\begin{theorem}\label{main-2}
If two normal binary delta-matroids $D$ and $\widetilde{D}$ have the same intersection graph, then $^{\partial}w_{D}(z)=~^{\partial}w_{\widetilde{D}}(z)$.
\end{theorem}

\begin{proof}
Since $G_{D}=G_{\widetilde{D}}$, $D=D(A_{G_{D}})$ and $\widetilde{D}=D(A_{G_{\widetilde{D}}})$, it follows that $D=\widetilde{D}$. Therefore $^{\partial}w_{D}(z)=~^{\partial}w_{\widetilde{D}}(z)$.
\end{proof}

\begin{theorem}\label{main-1}

Let $B$ and $\widetilde{B}$ be two bouquets. If $G_{D(B)}=G_{D(\widetilde{B})}$, then $^{\partial}\varepsilon_{B}(z)={^{\partial}}\varepsilon_{\widetilde{B}}(z)$.
\end{theorem}

\begin{proof}
If $G_{D(B)}=G_{D(\widetilde{B})}$, then $D(B)=D(\widetilde{B})$. For any $A\subseteq E(B)$, we denoted its corresponding
subset of $E(\widetilde{B})$ by $\widetilde{A}$.  By Lemma \ref{lem 1},
$$D(B^{A})=D(B)*A=D(\widetilde{B})*\widetilde{A}=D(\widetilde{B}^{\widetilde{A}}).$$
We have $w(D(B^{A}))=w(D(\widetilde{B}^{\widetilde{A}}))$.
Since $w(D(B^{A}))=\varepsilon(B^{A})$ and $w(D(\widetilde{B}^{\widetilde{A}}))=\varepsilon(\widetilde{B}^{\widetilde{A}})$, it follows that $\varepsilon(B^{A})=\varepsilon(\widetilde{B}^{\widetilde{A}}).$ Thus $^{\partial}\varepsilon_{B}(z)=~^{\partial}\varepsilon_{\widetilde{B}}(z).$
\end{proof}

\begin{lemma}\cite{QYXJ}\label{lem4}
Let $B_{t}$ be a bouquet with the signed rotation $$(1, 2, 3, . . . , t, 1, 2, 3, . . . , t).$$ We have
\begin{eqnarray*}
^{\partial}\varepsilon_{B_{t}}(z)=\left\{\begin{array}{ll}
                    2^{t}z^{t-1}, & \mbox{if}~t~\mbox{is odd}\\
                    2^{t-1}z^{t}+2^{t-1}z^{t-2}, & \mbox{if}~t~\mbox{is even.}
                   \end{array}\right.
\end{eqnarray*}
\end{lemma}

\begin{proposition}\label{pro 1}
Let $D$ be a normal binary delta-matroid and let $v$ be the number of vertices of $G_{D}$. If $G_{D}$ is a complete graph, then
\[^{\partial}w_{D}(z)=\left\{\begin{array}{ll}
2^{v}z^{v-1}, & \text{if $v$ is odd}\\

2^{v-1}z^{v}+2^{v-1}z^{v-2}, & \text{if $v$ is even}.
\end{array}\right.\]
\end{proposition}

\begin{proof}
By Theorem \ref{main-2}, we just need to find a normal binary delta-matroid $\widetilde{D}$ such that $G_{\widetilde{D}}=G_{D}$ and  the evaluation $^{\partial}w_{\widetilde{D}}(z)$ is easy to obtained. Let $B_{v}$ be a bouquet with the signed rotation $(1, 2, 3, . . . , v, 1, 2, 3, . . . , v).$
Obviously, $G_{D(B_{v})}=G_{D}=K_{v}$. By Lemma  \ref{lem 3} and Theorem \ref{main-2},
$$^{\partial}w_{D}(z)=~^{\partial}w_{D(B_{v})}(z)=~^{\partial}\varepsilon_{B_v}(z).$$
Therefore $^{\partial}w_{D}(z)$ can be obtained by Lemma \ref{lem4}.
\end{proof}

\begin{theorem}
Let $D=(E, \mathcal{F})$ be a delta-matroid. Then $^{\partial}w_{D}(z)=k$  for some integer $k$ if and only if $|\mathcal{F}|=1$.
\end{theorem}

\begin{proof}
The sufficiency is straightforward. To prove the necessity, suppose that $|\mathcal{F}|\geq 2$, then there exist $A_{1}, A_{2}\in \mathcal{F}$ such that $A_{1}\neq A_{2}$. Since $^{\partial}w_{D}(z)=k$,
we have $w(D)=0$. Thus $|A_{1}|=|A_{2}|$. Then for any $x\in A_{1}\backslash A_{2}$ we have $A_{1}\backslash x, A_{2}\cup x\in \mathcal{F}(D*x)$. Obviously, $|A_{1}\backslash x|\neq |A_{2}\cup x|$, a contradiction, since $~^{\partial}w_{D}(z)=~^{\partial}w_{D*x}(z)=k$ by Proposition \ref{pro 2}.
\end{proof}

\begin{lemma}\cite{CMNR}\label{lem 2}
Let $D=(E, \mathcal{F})$ be a delta-matroid and $A\subseteq E$. Then $$w(D|_A)=\rho_{D}(A)-r_{D_{min}}(A)-n_{D_{min}}(E)+n_{D_{min}}(A).$$
\end{lemma}

\begin{lemma}\label{lem 5}
Let $D=(E, \mathcal{F})$ be a normal delta-matroid and $A\subseteq E$. Then \[w(D*A)=w(D|_{A})+w(D|_{A^{c}}).\]
\end{lemma}

\begin{proof}
Since $D$ is a normal delta-matroid, it follows that $D_{min}=(E, \{\emptyset\})$. We have $r_{D_{min}}(A)=r_{D_{min}}(A^{c})=0$, $n_{D_{min}}(E)=|E|, n_{D_{min}}(A)=|A|$ and $n_{D_{min}}(A^{c})=|A^{c}|$.
Then by Lemma \ref{lem 2},
\begin{eqnarray*}
&~&w(D|_{A})+w(D|_{A^{c}})\\
&=&\rho_{D}(A)-r_{D_{min}}(A)-n_{D_{min}}(E)+n_{D_{min}}(A)+\\
&~&\rho_{D}(A^{c})-r_{D_{min}}(A^{c})-n_{D_{min}}(E)+n_{D_{min}}(A^{c})\\
&=&\rho_{D}(A)+\rho_{D}(A^{c})-|E|\\
&=&|E|-min\{|A\Delta F|: F\in \mathcal{F}\}+\\
&~&|E|-min\{|A^{c}\Delta F|: F\in \mathcal{F}\}-|E|\\
&=&|E|-min\{|A^{c}\Delta F|: F\in \mathcal{F}\}-min\{|A\Delta F|: F\in \mathcal{F}\}\\
&=&|E|-r((D*A^{c})_{min})-r((D*A)_{min})\\
&=&|E|-r((D*A)^{*}_{min})-r((D*A)_{min})\\
&=&r((D*A)_{max})-r((D*A)_{min})=w(D*A).
\end{eqnarray*}
\end{proof}

\begin{remark}
This is not right for non-normal delta-matroids. For example, let $D=(\{1 , 2\}, \{\{1\}, \{2\}\})$. It is easy to check that $D*1=(\{1 , 2\}, \{\emptyset, \{1, 2\}\})$, $D|_{1}=(\{1 \}, \{\{1\}\})$ and $D|_{2}=(\{2\}, \{\{2\}\})$. Obviously, $w(D*1)=2$ and $w(D|_1)=w(D|_2)=0$. Note that $w(D*1)\neq w(D|_{1})+w(D|_2)$.





\end{remark}

\begin{theorem}
Let $D=(E, \mathcal{F})$ be a binary normal delta-matroid. Then $^{\partial}w_{D}(z)$ contains non-zero constant term if and only if $G_{D}$ is a bipartite graph.
\end{theorem}

\begin{proof}
Since $^{\partial}w_{D}(z)$ contains non-zero constant term, it follows that $D$ is a twist of a matroid. On account of the property that twist preserving evenness, we have that $D$ is even and hence
$G_{D}$ is a simple graph. Suppose that $G_{D}$ is not bipartite. Then $G_{D}$ contains an odd cycle $P$  of length more than or equal to 3. We denote by $A$ the subset of $E$ corresponding to the vertices of $P$. It is obvious that deleting can not increase the width. Then for any subset $B$ of $E$, we have $w(D|_{B\cap A})\leq w(D|_{B})$ and $w(D|_{B^{c}\cap A})\leq w(D|_{B^{c}})$.
Since $D$ is a normal binary delta-matroid, we know that $D=D(C)$ for some symmetric matrix $C$ over $GF(2)$.
Note that there are $e, f\in B\cap A$ or $e, f\in B^{c}\cap A$ such that
\[C[\{e, f\}] =
\bordermatrix{
& e & f   \cr
e & 0   & 1     \cr
f  & 1   & 0     \cr
}.\]
Then $$w(D|_{B})+w(D|_{B^{c}})\geq w(D|_{B\cap A})+w(D|_{B^{c}\cap A})>0.$$ Thus  by Lemma \ref{lem 5} $w(D*B)=w(D|_{B})+w(D|_{B^{c}})>0,$ a contradiction.

Conversely, if $G_{D}$ is bipartite and non-trivial, then its vertex set can be partitioned into two subsets $X$ and $Y$ so that every edge of $G_{D}$ has one end in $X$ and the other end in $Y$.
For these two subsets $X$ and $Y$ of the vertex set of $G_{D}$, we denoted these two corresponding subset of $E$ also by $X$ and $Y$. Obviously, $X\cup Y=E$, $X\cap Y=\emptyset$ and $w(D|_{X})=w(D|_{Y})=0.$
Thus $w(D*X)=w(D|_{X})+w(D|_{Y})=0$ by Lemma \ref{lem 5}. Hence $^{\partial}w_{D}(z)$ contains non-zero constant term.
\end{proof}


\begin{theorem}\label{main-3}
Let $D=(E, \mathcal{F})$ be a connected even normal binary delta-matroid. Then $^{\partial}w_{D}(z)=mz^k$ if and only if $G_{D}$ is a complete graph of odd order.
\end{theorem}

\begin{proof}
The sufficiency is easily verified by Proposition \ref{pro 1}. For necessity,
\[D=\left\{\begin{array}{ll}
(\{e\}, \{\emptyset\}), & \text{if $|E|=1$}\\

(\{e, f\}, \{\emptyset, \{e, f\}\}), & \text{if $|E|=2$}.
\end{array}\right.\]
Then the result is easily verified when $|E|\in \{1, 2\}$. Assume that $|E|\geq 3$. Let $e, f, g\in E$. We consider three claims:

{\bf Claim 1.} If $\{e, f\}\in \mathcal{F}$, we have $r_{{D^{*}}_{min}}(\{e, f\})=1$ and $r_{{D^{*}}_{min}}(e)=1$.

{\bf Proof of Claim 1.} For any $A\in \mathcal{F}_{max}$, we observe that $A\cap \{e, f\}\neq \emptyset$. Otherwise, $\emptyset, A\cup \{e, f\}\in \mathcal {F}(D*\{e, f\})$. Since $\{e, f\}\in \mathcal{F}$, it follows that $\emptyset\in \mathcal {F}(D*\{e, f\})$. Then $w(D*\{e, f\})> w(D)$, this contradicts $^{\partial}w_{D}(z)=mz^k$. Furthermore, we observe that there exists $B\in \mathcal{F}_{max}$ such that $e\notin B$. Otherwise, $r(D*e_{max})=r(D_{max})-1$. Since $e\notin \mathcal{F}$ and $\emptyset \in \mathcal{F}$ , we have $r(D*e_{min})=1$. Then $w(D*e)=w(D)-2$, this also contradicts $^{\partial}w_{D}(z)=mz^k$. Consequently, for any $A\in {\mathcal{F}^{*}}_{min}$, $A\cap \{e, f\}\neq \{e, f\}$ and there exists $B\in {\mathcal{F}^{*}}_{min}$ such that $e\in B$. Thus $r_{{D^{*}}_{min}}(\{e, f\})=1$ and $r_{{D^{*}}_{min}}(e)=1$.

{\bf Claim 2.} If $\{e, f\}\notin \mathcal{F}$, we have $r_{{D^{*}}_{min}}(\{e, f\})=2$.

{\bf Proof of Claim 2.} There exists $A\in \mathcal{F}_{max}$ such that $\{e, f\}\cap A=\emptyset$. Otherwise, $r(D*\{e, f\}_{max})\leq r(D_{max})$. Since $\{e, f\}\notin \mathcal{F}$ and $\emptyset \in \mathcal{F}$ , we have $r(D*\{e, f\}_{min})=2$. Then  $w(D*\{e, f\})\leq w(D)-2$, a contradiction.
Consequently, there exists $A\in {\mathcal{F}^{*}}_{min}$ such that $\{e, f\}\in A$. Thus $r_{{D^{*}}_{min}}(\{e, f\})=2$.

{\bf Claim 3.} $E$ does not contain $e, f, g$ such that $\{e, f\}, \{e, g\}\in \mathcal {F}$ and $\{f, g\}\notin \mathcal{F}$.

{\bf Proof of Claim 3.} Assume that Claim 3 is not true. Since $\{e, f\}, \{e, g\}\in \mathcal {F}$, it follows that $r_{{D^{*}}_{min}}(\{e, f\})=1$ and $r_{{D^{*}}_{min}}(\{e, g\})=1$ by Claim 1. Then
\begin{eqnarray*}
&~&r_{{D^{*}}_{min}}(\{e, f\}\cup\{e, g\})+r_{{D^{*}}_{min}}(\{e, f\}\cap\{e, g\})\\
&=&r_{{D^{*}}_{min}}(\{e, f, g\})+r_{{D^{*}}_{min}}(e)\\
&\leq&r_{{D^{*}}_{min}}(\{e, f\})+r_{{D^{*}}_{min}}(\{e, g\})\\
&=&2.
\end{eqnarray*}
Hence $r_{{D^{*}}_{min}}(\{e, f, g\})\leq 1$. Since $\{f, g\}\notin \mathcal{F}$, we have $r_{{D^{*}}_{min}}(\{f, g\})=2$ by Claim 2. But $$r_{{D^{*}}_{min}}(\{f, g\})\leq r_{{D^{*}}_{min}}(\{e, f, g\})\leq1,$$ a contradiction.

Suppose that $G_{D}$ is not a complete graph. Note that $G_{D}$ is connected. Then there is a vertex set $v_{e}, v_{f}, v_{g}$ of $G_{D}$ such that the induced subgraph $G_{D}(\{v_{e}, v_{f}, v_{g}\})$ is a 2-path. We may assume without loss of generality that the degree of $v_{e}$ is 2 in $G_{D}(\{v_{e}, v_{f}, v_{g}\})$. Since $D$ is a normal binary delta-matroid, we know that $D=D(C)$ for some symmetric matrix $C$ over $GF(2)$.
Then
   \[C[\{e, f, g\}] =
\bordermatrix{
& e & f & g  \cr
e & 0   & 1   & 1  \cr
f  & 1   & 0   & 0  \cr
g & 1   & 0   & 0  \cr
}.\]
Thus $\{e, f\}, \{e, g\}\in \mathcal {F}$ and $\{f, g\}\notin \mathcal{F}$, this contradicts Claim 3. Therefore $G_{D}$ is a complete graph.
\end{proof}

\begin{corollary}
Let $D=(E, \mathcal{F})$ be an even normal binary delta-matroid. Then $^{\partial}w_{D}(z)=mz^k$ if and only if each connected component of $G_{D}$ is a complete graph of odd order.
\end{corollary}
\begin{proof}
The proof is straightforward by Proposition \ref{pro 2} and Theorem \ref{main-3}.
\end{proof}


\section*{Acknowledgements}
This work is supported by NSFC (No. 11671336) and the Fundamental Research Funds for the Central Universities (Nos. 20720190062, 2021QN1037).


\end{document}